\documentclass[sn-mathphys-num]{sn-jnl}

\usepackage{graphicx}%
\usepackage{multirow}%
\usepackage{amsmath,amssymb,amsfonts}%
\usepackage{amsthm}%
\usepackage{xcolor}%
\usepackage{textcomp}%
\usepackage{manyfoot}%
\usepackage{booktabs}%
\usepackage{algorithm}%
\usepackage{algorithmicx}%
\usepackage{algpseudocode}%
\usepackage{listings}%
\usepackage{lmodern}
\usepackage{enumerate}

\newtheorem{thm}{Theorem}[section]
\newtheorem{thmx}{Theorem}
\newtheorem{cor}[thm]{Corollary}
\newtheorem{lem}{Lemma}[section]

\newtheorem{case}{Case}
\newtheorem{subcase}{Subcase}

\newtheorem*{conj*}{Denjoy's Conjecture}

\theoremstyle{definition}

\newtheorem{rem}{Remark}[section]
\newtheorem{Que}{Question}
\newtheorem{exa}{Example}[section]
\numberwithin{equation}{section}
\begin{document}

\title{A study on entire functions sharing one   function with their difference operators   and its application}

\author[1]{\fnm{Xuxu} \sur{Xiang}}\email{xiangxuxu@gznu.edu.cn}

\author*[1]{\fnm{Jianren} \sur{Long}}\email{longjianren2004@163.com}

\affil[1]{\orgdiv{School of Mathematical Sciences}, \orgname{Guizhou Normal University}, \orgaddress{\city{Guiyang}, \postcode{550025}, \country{ P. R. China}}}

\abstract{
	
Let $f$ be a transcendental entire function with hyper-order strictly less than 1 and having a Borel exceptional small function. If $f$ and $\Delta^n f$, or $f'$ and $f(z+1)$, share a function CM, then the exact form of $f$ is determined, which improves the previous results given by L\"u et al. [Results Math. 74, article number 30 (2019)] and Liu et al. [Bull. Korean Math. Soc. 51, 1453-1467 (2014)]. As an application, the relationship between $f$ and $\Delta^n f$ is established under the condition that they share a finite set, which partially resolves Liu's question raised in [J. Math. Anal. Appl. 359, 384-393 (2009)]. Furthermore, several  examples are presented to demonstrate these results.
 }



\maketitle

\section{Introduction}\label{sec1}
Let $f$ be a meromorphic function in the complex plane $\mathbb{C}$. Assume that the reader is familiar with the standard notation and basic results of Nevanlinna theory, such as $m(r,f),~N(r,f)$,$~T(r,f)$, see \cite{Hayman} for more details.  A meromorphic
function $g$ is said to be a small function of $f$ if $T(r,g)=S(r,f)$, where $S(r,f)$  denotes any quantity
that satisfies $S(r,f)= o(T(r, f))$ as $r$ tends to infinity, outside a possible exceptional set of finite linear measure. $\rho(f)=\underset{r\rightarrow \infty}{\lim\sup}\frac{\log^+T(r,f)}{\log r}$ and   $\rho_2(f)=\underset{r\rightarrow \infty}{\lim\sup}\frac{\log^+\log^+T(r,f)}{\log r}$  are used to denote the order and the hyper-order  of $f$, respectively. Define $\lambda(f)$ as the exponent of convergence of the zeros sequences of $f$, and define $\mu(f)$ as the lower order of $f$.
 The $n$th difference
operator of $f$ is defined by
$\Delta_c^n f(z) = \Delta_c^{n-1} \big( \Delta_c f(z) \big) = \sum_{i=0}^{n} (-1)^{n-i} \binom{n}{i} f(z + ic)$, where $c$ is a constant.

Due to Nevanlinna's second main theorem, Nevanlinna \cite{Nevanlinna} established the five-value theorem: If two non-constant meromorphic functions $f$ and $g$ share five distinct values IM (ignoring multiplicities), then $f = g$.
If $f$ and $g$ share four distinct values CM (counting multiplicities), according to Nevanlinna four-value theorem \cite{Nevanlinna}, it implies that $f$ can be transformed into $g$ through a M\"obius transformation.
Furthermore, the assumption of 4 CM in Nevanlinna four-value theorem has been improved to 2 CM + 2 IM by Gundersen \cite{Gundersen1}. However, the assumption of 4 CM cannot be improved to 4 IM \cite{Gundersen2}. For further details, refer to \cite{yang2003}.

The number of shared values can be reduced when $f$ and $g$ are related. Rubel et al. \cite{Rubel} showed that if a non-constant entire function $f$ and its first derivative $f'$ share two distinct values CM, then they are identical.
Subsequently, Mues et al. \cite{Mues} and Gundersen \cite{Gundersen} extended this result to meromorphic functions.
Regarding the case where $f$ and $f'$ share one finite value CM, Br\"uck\cite{B} raised the following conjecture.

\vspace{4pt}

 \noindent{\bf Br\"uck conjecture.}  Let $f$ be a non-constant entire function with finite hyper-order $\rho_2(f)\notin \mathbb{N}$. If $f$ and $f'$ share one finite
value $a$ CM, then $f-a=c(f'-a)$, where  $c$ is a non-zero constant.
\vspace{4pt}

The Br\"uck conjecture has not been fully solved up to the present. In the case where $f$ is of finite order, the Br\"uck conjecture has been resolved by Gundersen et al. \cite{Gundersen3}.  Cao \cite{Cao} showed that the Br\"uck conjecture is also true when $\rho_2(f) < \frac{1}{2}$.

Recently, the difference analogues of Nevanlinna theory have been established \cite{Chiang,Halburd,Halburd1}. These analogues serve as a powerful theoretical tool for studying the uniqueness problems of meromorphic functions, considering their shifts or delay-differential, see \cite{H1,H2}, \cite[Chapter 11]{Chen1} and \cite[Chapter 3]{Liu1} and references therein.

The paper is organized as follows. Section \ref{sec2} determines the exact form of $f$ when $f$ and $\Delta^n f$ share a function CM, while Section \ref{sec4} deals with the case where $f'$ and $f(z+1)$ share a function CM.  In Section \ref{sec3}, we apply the results from Sections \ref{sec2} and \ref{sec4} to establish two relationships:
(i) between an entire function \(f\) and \(\Delta^n f\), and
(ii) between \(f'\) and \(f(z+c)\),
provided that each pair shares a finite set.

\section{\texorpdfstring{$f$ and $\Delta^n f$ share an entire function}{f and Delta^n f share an entire function}}\label{sec2}

In 2014, Chen et al. \cite{Chen}  studied the difference analogues of the Br\"uck conjecture and got the following result.

\vspace{4pt}

\begin{thmx} \cite{Chen}
\label{tha}
	Let \( f \) be a finite order transcendental entire function such that
\(\lambda(f - a) < \rho(f)\), where \( a \) is an entire function with \(\rho(a) < 1\).
Let \( n \) be a positive integer.
If \(\Delta^n f\) and \( f \) share an entire function \( b \) (\( b \neq a \) and \(\rho(b) < 1\)) \text{CM},
where \(\Delta^n f \neq 0\), then
\[
f = a + c e^{c_1 z},
\]
where \( c, c_1 \) are two non-zero constants.
	
\end{thmx}

Later,  Li et al. \cite{Li} considered the case $a = b$ in Theorem  \ref{tha} and obtained
the following theorem.

\vspace{4pt}

\begin{thmx}\cite{Li}
\label{thb}
	Let \( f \) be a finite order transcendental entire function such that
\(\lambda(f - a) < \rho(f)\), where \( a \neq 0 \) is an entire function with \(\rho(a) < 1\).
Then, we have \( f - a \) and \(\triangle_\eta^n f - a \) share \(0\) \text{CM} if and only if
$f = a + B \left[ \triangle_\eta^n a - a \right] e^{Az}$
and
$\triangle_\eta^{2n} a - \triangle_\eta^n a = 0,$
where \( A, B \) are non-zero constants with \( e^{A\eta} = 1 \).	
\end{thmx}

\vspace{4pt}

  L\"u et al. \cite{Lu}  given a joint theorem involve of both cases $a = b$ and $a \not= b$ , and the condition $\rho(b)<1$ in Theorem \ref{tha} is  weakened.

\vspace{4pt}

\begin{thmx} \cite{Lu}
\label{thc}
	Let \( f \) be a finite order transcendental entire function and let \( n \) be a positive integer. 
Assume that $a$ and $b$ are both entire functions such that   $ \lambda(f - a) < \rho(f)$, $\rho(a)<1$, $\rho(a) \neq \rho(f)$ and  $\rho(b) < \max\{1, \rho(f)\}$. 
If \(\triangle^n f-b\) and \( f -b\) share \(0 \) \text{CM}, where \( b \) is an entire function with \(\rho(b) < \max\{1, \rho(f)\}\), then
\[
f = a + c e^{\gamma z},
\]
where \( c, \gamma \) are two non-zero constants.
In particular, if \( a = b \), then \( a \) reduces to a constant.
	
\end{thmx}

 The following example is not covered by Theorem \ref{thc} when $\Delta^n f \equiv 0$.
\begin{exa}
\label{e1}
$f=ze^{2\pi iz}+z$, $a=b=z$, then $\Delta^2f=0$, $\Delta^2f-z$ and $f-z=ze^{2\pi iz}$ share 0 \text{CM}.
\end{exa}

Note that the function $f $is finite order in the Theorems A-C. It is known that the difference analogues
to Nevanlinna theory given by Halburd et al. \cite{Halburd} and  Chiang et al. \cite{Chiang} has been improved by Halburd et al. \cite{Halburd1} from the finite order of meromorphic functions to  hyper order
strictly less than 1. So, one may ask the question:

\vspace{4pt}

\begin{Que}
\label{q1}
Whether the theorem \ref{thc} still holds if the entire function $f $ is of  hyper order strictly less than 1 and the conditions on 
$a,b $ are weakened?
\end{Que}

    \vspace{6pt}
We give an answer to the Question \ref{q1} by proving the following
theorem.
\vspace{6pt}

\begin{thm}
\label{th1.1}
	Let \( f \) be a  transcendental entire function with $\rho_2(f)<1$, and let \( n \) be a positive integer. 
Assume that $a$ and $b$ are both  small entire functions such that $\lambda(f - a) < \rho(f)$ and $\rho(a) < \rho(f)$.
If \(\Delta^n f-b\) and \( f -b\) share \(0 \) \text{CM}, then
one of the following assertions holds.
	\begin{enumerate}
		\item[\rm{(i)}]If $\Delta^nf\not\equiv0$,   then
$f(z) = a(z) + pe^{\gamma z},$
 and $\Delta^n f-b=q(f-b)$, where \( p, \gamma, q \) are  non-zero constants and $(e^{\gamma}-1)^n=q.$ In particular, if \( a = b \), then \( a=0 \).
		\item[\rm{(ii)}] If $\Delta^nf\equiv0$ and $\rho(a)<1$, then $f = a(z) + Aa(z) e^{\gamma z},$ $a=b$, $e^{\gamma}=1$, $\Delta^na\equiv0$ and  $\Delta^n f-b=\frac{-1}{Ae^{\gamma z}}(f-b)$, where $a,b$ are polynomials and $A$, $ \gamma$ are non-zero constants.
	\end{enumerate}
\end{thm}

\begin{rem}
\label{r1}

  Compared to Theorem~\ref{thc}, the condition on $a$ and $b$ in our theorem  differs in that we require $a$ and $b$ both to be  small functions  of $f$  such $\rho(a) < \rho(f)$ without   assuming $\rho(a) < 1$, $\rho(a) \neq \rho(f)$ and $\rho(b) < \max\{1, \rho(f)\}$. In fact, under the assumption that $a$ and $b$ are both entire functions satisfying $\lambda(f - a) < \rho(f)$, the conditions in Theorem~C that $\rho(a) \neq \rho(f)$, $\rho(a) < 1$ and $\rho(b) < \max\{1, \rho(f)\}$ imply that $a$ and $b$ are both small entire functions of $f$ and $\rho(a) < \rho(f)$. In the following, we present the proof.
  
   Assume that $a$ and $b$ are both entire functions such that $\lambda(f - a) < \rho(f)$, $\rho(a)<1$, $\rho(a) \neq \rho(f)$ and $\rho(b) < \max\{1, \rho(f)\}$. We claim $\rho(f) \ge 1$. Otherwise, assume $\rho(f) < 1$. Noting that $\rho(a) \neq \rho(f)$, we have $\rho(f - a) = \max\{\rho(a), \rho(f)\}$. It follows that $\lambda(f - a) < \rho(f)\le \rho(f - a)$, which implies $f - a$ has two Borel exceptional values, namely $0$ and $\infty$. Consequently, $f - a$ must be of regular growth with $\rho(f - a) \geq 1$, leading to a contradiction. Hence, $\rho(f) \ge 1$.
    
Now, since $\lambda(f - a) < \rho(f) = \rho(f - a)$, we deduce that $f - a$ is of regular growth with $\rho(f - a) = \mu(f - a) \geq 1$. Since $\mu(a) < 1$, we get $\mu(f) \geq 1$. Moreover, from $\rho(b) < \max\{1, \rho(f)\}$ and $\rho(a) < 1$, we obtain $\rho(b) < \rho(f) = \mu(f)$ and $\rho(a) < \mu(f)$. This confirms that $a$ and $b$ are small functions of $f$ and $\rho(a)<\rho(f)$.
\end{rem}

\vspace{6pt}

\begin{rem}
Example \ref{e1} illustrates that case (ii) of Theorem \ref{th1.1} may occur. The case (i) of Theorem \ref{th1.1} gives an affirmative answer to the Question \ref{q1}.
\end{rem}

\vspace{6pt}
The following example illustrates that case (i) of Theorem \ref{th1.1} may occur.
\vspace{6pt}

\begin{exa}\cite[Remark 6]{Lu}
Let \(a(z)=z\), \(b(z)=\frac{(e - 1)z-1}{e - 2}\) and \(f(z)=a(z)+e^{z}=z + e^{z}\). Then
$
\frac{\Delta f - b}{f - b}=e - 1.
$
 Obviously, \(\Delta f - b\) and \(f - b\) share \(0\) CM.
\end{exa}

The following example  is given to show that the condition $\rho(a)<\rho(f)$ is necessary.

\vspace{4pt}

\begin{exa}
    Consider \( f(z) = e^{z\ln2} (e^{2\pi i z} + e^{6k\pi iz}) \). Obviously, \( \Delta f(z) = f(z) \). Assume that \( b \) is a arbitrary entire functions of order less than 1. Then \( \Delta f(z) \) and \( f(z) \) share  \( b \) CM.  However, $e^{ln2z}e^{6k\pi iz}$ is a Borel exceptional function of $f$,  but the form of \( f \) does not satisfy the conclusion of Theorem \ref{th2.1}.
\end{exa}

\subsection{Lemmas}

Before proving Theorem \ref{th1.1}, we need some lemmas.
The following two lemmas are Borel type theorem, which can be found in \cite{yang2003}.

\vspace{6pt}

\begin{lem}\cite[Theorem 1.51]{yang2003}
	\label{lm1}
	Let $f_1,\,f_2,\ldots,f_n~(n\ge2)$ be meromorphic functions, $g_1,\,$
	$g_2,\ldots,g_n$ be entire functions satisfying
	the following conditions,
	\begin{itemize}
		\item [\rm{(i)}]$\sum\limits_{j=1}^n{{f_j(z)}}e^{g_j(z)}\equiv0$,
		\item[\rm{(ii)}]for $1\le{j}<{k}\le{n}$, $g_j-g_k$ is not constant,
		\item[\rm{(iii)}]for $1\le{j}\le{n},1\le{t}<{k}\le{n}$, $T(r,f_j)=o(e^{g_t-g_k})$, $r\rightarrow\infty,\,r\notin{E}$, where $E$ is the set of finite linear measure.
	\end{itemize}
	Then $f_j(z)\equiv0,\, j=1,\ldots,n$.
\end{lem}
\vspace{6pt}

\begin{lem}\cite[Theorem 1.62]{yang2003}
\label{lm2}
Let $f_j~(j=1,2,\dots,n)$ be meromorphic functions and $f_k~(k=1,2,\dots,n-1)$ be non-constants. If $n\ge3$,
\begin{align*}
 \sum\limits_{j=1}^n{f_i}\equiv1 ,\,
 \sum\limits_{j=1}^n{N\left(r,\frac{1}{f_j}\right)}+
(n-1)\sum\limits_{j=1}^n\overline{N}(r,f_j)<(\lambda+o\,(1))T(r,f_k),r\notin{E},
\end{align*}
where $\lambda<1$, E is the set of finite
linear measure, then $f_n\equiv1$.
\end{lem}

\vspace{6pt}
  Chiang et al. \cite{Chiang1} established complete asymptotic relations of difference quotients for finite order meromorphic functions as follows:

 \vspace{6pt}

\begin{lem}\cite[Theorem 5.1]{Chiang1}
\label{lm3}
Let \(f\) be a non-constant meromorphic function of finite order \(\rho<1\) and \(\eta\in\mathbb{C}\). Then for any given \(\epsilon>0\), and integers \(0\leq j < k\), there exists a set \(E\subset[1, +\infty)\) of finite logarithmic measure, so that for all \(|z|\not\in E\cup[0,1]\), we have
\[
\left|\frac{\Delta_{\eta}^{k}f(z)}{\Delta_{\eta}^{j}f(z)}\right|\leq |z|^{(k - j)(\rho - 1+\epsilon)}.
\]
\end{lem}

\begin{lem}\cite[Theorem 11.4.2]{Chen1}
\label{chen}
Let $f(z)$ be a finite order transcendental entire function such that 
$\lambda(f - a(z)) < \rho(f)$, where $a(z)$ is an entire function and satisfies $\rho(a) < 1$. 
Let $n$ be a positive integer. If $\Delta_\eta^n f(z)$ and $f(z)$ share $a(z)$ \textup{CM}, 
where $\eta \in \mathbb{C}$ satisfies $\Delta_\eta^n f(z) \not\equiv 0$, then
\[
a(z) \equiv 0 \quad \text{and} \quad f(z) = c e^{c_1 z},
\]
where $c, c_1$ are two nonzero constants.
\end{lem}



\vspace{6pt}
In proving Theorem \ref{th1.1}, we need growth estimates for solutions of some  difference equations.
\vspace{6pt}

\begin{lem}\cite[Corollary 2.2]{Lu}
\label{lm2.4}
Let $f$ be a transcendental entire function such that $\Delta^n f\not\equiv 0$, $\lambda(f - a) < \rho(f) = \rho < \infty$ and $\rho > 1$, where $a$ is an entire function with $\rho(a) < \rho$, and let $b$ be an entire function such that $\rho(b) < \rho$. Suppose that $f$ is a solution of the difference equation
\begin{equation*}
\Delta^n f - b = (f - b) e^Q 
\end{equation*}
where $Q$ is a polynomial. Then $\deg Q = \rho(f) - 1$.
\end{lem}

\begin{lem}\cite[Corollary 3.2]{Lu}
\label{lm2.7}
Let $a_0, a_1, \dots, a_n$ be constants satisfying $a_0 a_n \neq 0$. If $f$ is a nonconstant meromorphic solution of the difference equation
\begin{equation*}
a_n f(z + n) + \cdots + a_1 f(z + 1) + a_0 f(z) = P(z),  
\end{equation*}
where $P$ is a polynomial, then either $\rho(f) \geq 1$ or $f$ is a polynomial. In particular, if $a_n \neq \pm a_0$, then $\rho(f) \geq 1$.
\end{lem}

\begin{lem}
\label{lm5}
Let \( f \) be a  transcendental entire function with $\rho_2(f)<1$, and let \( n \) be a positive integer. 
Assume that $a$ and $b$ are both  small entire functions such that $\lambda(f - a) < \rho(f)$, $\rho(a) < \rho(f)$. If $f$ is a solution of \begin{align}
\label{2.1}
    \Delta^n f - b = e^{Q}(f - b),
\end{align}
where   \(Q\) is an entire function, then \(f = a + P e^{h}\), where \(P\) is an entire function with \(\rho(P)<\rho(e^h)\) and \(h\) is a non-constant polynomial. In particular, if $\Delta^nf\equiv0$ and $\rho(a)<1$, then  $f(z)= a(z) + Aa(z) e^{c_1 z},$ $a=b$, $\Delta^{n}a\equiv0$, and $e^{c_1}=1$, where $a,b$ are polynomials and   $-A=\frac{1}{e^{c_1z}e^Q}$, $c_1$ are non-zero constants.
\end{lem}

\begin{proof}[Proof of  Lemma \ref{lm5}] 
Since  $ \lambda(f - a) < \rho(f)$, we get $f-a=Pe^{h}$, where $h$ and $P$ are entire functions such that $\rho(P)<\rho(f)$. It fallows that $a$ is a small function of $f$, we get $T(r,f-a)=(1+o(1))T(r,f)\le T(r,P)+T(r,e^h)$.  Therefore, $\rho(f-a)=\rho(f)=\rho(Pe^h)\le \max\{\rho(P),\rho(e^h)\}$. By $\rho (P)<\rho(f)$, we get $\rho (f)\le \rho(e^h)$. 
 However, $\mu(e^h)=\rho(e^h)$, thus $\rho(P)<\mu(e^h)$, which implies $P$ is a small function of $e^h$. From $f-a=Pe^h$, where $a$ is a small function of $f$ and $P$ is a small function of $e^h$, we get $(1+o(1))T(r,f)=(1+o(1))T(r,e^h)$. 
Thus,  $\rho(f)=\mu(f)$ and  $P,b$ and $a$ are small functions of $e^h$.

Substituting the form of \(f\) into \eqref{2.1} yields
\begin{align}
\label{2.2}
\frac{\sum_{j = 0}^{n} C_{n}^{j} (-1)^{n - j} P(z + j) e^{h(z + j)}+\Delta^{n}a - b}{Pe^{h}+a - b}=e^{Q}.
\end{align}

\begin{case}\rm{\(\Delta^n f\equiv0\). 
 
 From \eqref{2.2}, we get
\begin{align}
\label{2.3}
\frac{-b}{Pe^h+a-b}=e^{Q}.
\end{align}
If $a\not\equiv b$, then by the  Nevanlinna second main theorem and \eqref{2.3}, we get
\begin{align*}
    T(r,Pe^h)< N(r,\frac{1}{Pe^h})+N(r,Pe^h)+N(r,\frac{1}{Pe^h+a-b})+S(r,Pe^h)\le  S(r,Pe^h).
\end{align*}
This is impossible.    Therefore, $a\equiv b$. 
By \eqref{2.3} and $a\equiv b$, we have the zero of $P$ must the zero of $b$, which yields $\lambda(P)=\rho(P)<1$.

Due to \(\Delta^n f\equiv0\), 
we deduce that  
\begin{align}
\label{2.4}
\sum_{j = 0}^{n} C_{n}^{j} (-1)^{n - j} P(z + j) e^{h(z + j)}+\Delta^{n}a\equiv0.
\end{align}
Since $\rho(P)<1$ and $\rho(a)<1$,
we claim that  there exist $i,k~ (0\le i<k\le n)$  such that $h(z+k)-h(z+i)=d$, where $d$ is a constant.  Otherwise, using Lemma \ref{lm1} to \eqref{2.4}, we get $P\equiv0$, which is impossible.  Differentiating    $h(z+k)-h(z+i)=d$ yields $h'(z+k)-h'(z+i)=0$. If $h$ is transcendental, then $h'$ is  a periodic function with period $i-k$. Hence, $\rho (h)\ge 1$. This contradicts $\rho (h)<1$.  Therefore, $h$ is a polynomial.

$h(z+k)-h(z+i)=d$ yields  $h$ is a polynomial of degree 1. Without loss of generality, we may assume that $h=c_1z$, where $c_1$ is a nonzero constants.
From \(\Delta^n f\equiv0\), we get 
    $e^h(z)(\sum_{j = 0}^{n} C_{n}^{j} (-1)^{n - j} P(z + j) e^{h(z + j)-h(z)})+\Delta^{n}a\equiv0.$
 It implies $\sum_{j = 0}^{n} C_{n}^{j} (-1)^{n - j} P(z + j) e^{c_1j}\equiv0$ and $\Delta^{n}a\equiv0$. By  $\rho(P)<\rho(f)=1$ and Lemma \ref{lm2.7},  we get  $P(z)$ is a polynomial.
 It leads to
\[
0=\sum_{j = 0}^{n}(-1)^{n - i}C_{n}^{j}(e^{ c_1})^{j}=(e^{c_1}-1)^{n},
\]
then $e^{c_1}=1$. By \eqref{2.3} we get $-Aa=P$, where $-A=\frac{1}{e^Qe^{c_1z}}$ is a non-zero constant.

}
\end{case}

\begin{case}\(\Delta^n f\not\equiv0\). \rm{

We will use proof by contradiction, so let's assume without loss of generality that $h$ is a transcendental entire function.
Let $w(z)=\sum_{j = 0}^{n} C_{n}^{j} (-1)^{n - j} P(z + j) e^{h(z + j)-h(z)}$. We can  rewrite \eqref{2.2} as
\begin{align}
\label{2.5}
\frac{we^h+\Delta^n a-b}{Pe^h+a-b}=e^Q.
\end{align}

If $w\equiv0$, then
 \begin{align}
\label{2.51}
\sum_{j = 0}^{n} C_{n}^{j} (-1)^{n - j} P(z + j) e^{h(z + j)}=0.
 \end{align}
  We claim that  there exist integers $i_0,k_0~ (0\le i_0<k_0\le n)$  such that $h(z+k_0)-h(z+i_0)$ is a polynomial.  Otherwise, $h(z+k)-h(z+l)$ is transcendental for any $m,k~ (0\le l<k\le n).$ Thus, $\rho(P)<\mu(e^{h(z+k)-h(z+l)})=\infty$, which implies $P$ is the small function of $e^{h(z+k)-h(z+l)}$. 
 Using Lemma \ref{lm1} to \eqref{2.51}, we get $P\equiv0$, which is impossible. Therefore, $h(z+k_0)-h(z+i_0)$
 is a polynomial of degree $m$. We deduce $h^{(m+1)}(z+k_0)-h^{(m+1)}(z+i_0)\equiv 0$ and $h^{m+1}$ is a periodic function.  Hence, $\rho (h)\ge 1$, this contradicts $\rho (h)<1$.  Thus, $w\not\equiv0$
 
 By difference logarithmic derivative lemma \cite{Halburd1},   we get that  $w$ is a small function of $e^h$.  Equation   \eqref{2.5} yileds
\begin{align}
\label{2.6}
\frac{(e^h+\frac{\Delta^na-b}{w})w}{(e^h+\frac{a-b}{P})P}=e^Q.
\end{align}

			\setcounter{subsection}{2}
			\setcounter{subcase}{0}
			\renewcommand{\thesubcase}{\arabic{subsection}.\arabic{subcase}}

\begin{subcase}
\rm{
$a\equiv b$.
 
 If $\Delta^na-b\not\equiv0$, by \eqref
 {2.6}  and the Nevanlinna  second main theorem, we get  $T(r,e^h)< N(r,\frac{1}{e^h+\frac{\Delta^na-b}{w}})+S(r,e^h)\le N(r,\frac{1}{P})+S(r,e^h)$, which is impossible.   Therefore, $\Delta^na\equiv b\equiv a$. Equation \eqref{2.6} implies $e^Q=\frac{w}{P}$, which is 
\begin{align}
\label{2.7}
\sum_{j = 1}^{n} C_{n}^{j} (-1)^{n - j} {\frac{P(z + j)}{P(z)}}   e^{h(z + j)-h(z)}+(-1)^n=e^Q.
\end{align}


 We claim that  there exist an integer $k$ $(0<k\le n)$  such that $h(z+k)-h(z)$ is a polynomial.  Otherwise,  we suppose that $h(z+j)-h(z)$ is transcendental for any integer $j$ satisfies $0< j \le n.$
Since $h(z+j)-h(z)$ is 
transcendental, we obtain $\rho(\frac{P(z+j)}{P(z)})<\mu(e^{h(z+j)-h(z)})=\infty$. Therefore, $\frac{P(z+j)}{P(z)}$ is a small function of  $e^{h(z+j)-h(z)}$. 
 
If $n=1$, then \eqref{2.7} reduces in to $\frac{P(z+1)}{P(z)}e^{h(z+1)-h(z)}-1=e^Q$. 
Using the Nevanlinna second main theorem, we can get 
$T(r,\frac{P(z+1)}{P(z)}e^{h(z+1)-h(z)})<\frac{1}{\frac{P(z+1)}{P(z)}e^{h(z+1)-h(z)}-1}+S(r,\frac{P(z+1)}{P(z)}e^{h(z+1)-h(z)})=S(r,\frac{P(z+1)}{P(z)}e^{h(z+1)-h(z)})$, which is
a contradiction. 

  If $n>1$, using Lemma \ref{lm2} to \eqref{2.7} implies $e^Q=(-1)^n$. 
  Thus,  \eqref{2.7} reduces in to 
\begin{align}
\label{2.71}
\sum_{j = 1}^{n} C_{n}^{j} (-1)^{n - j} {\frac{P(z + j)}{P(z)}}   e^{h(z + j)-h(z)}=0.
\end{align}
By using the method of analyzing equation \eqref{2.51} to \eqref{2.71} , we can similarly obtain  a contradiction. 

 Therefore,  there exist $k$  such that $h(z+k)-h(z)$ is a polynomial.
We deduce   $\rho (h)\ge 1$, this contradicts $\rho (h)<1$.   
 }
\end{subcase}

\begin{subcase} \rm{
$a\not\equiv b$.
  By the Nevanlinna  second main theorem, we get $N(r,\frac{1}{e^h+\frac{a-b}{P}})=O(T(r,f))$. From \eqref{2.6}, we find that the zero of $e^h+\frac{a-b}{P}$ must be the zero of $w$ and $e^h+\frac{\Delta^na-b}{w}$. We denote by \(N_1\) the  counting function of those common zeros of $e^h+\frac{a-b}{P}$  and $e^h+\frac{\Delta^na-b}{w}$. Since $w$ is a small function of $e^h$,  then $N_1=O(T(r,f))$.

  Let $z_0$ be a zero of $e^h+\frac{a-b}{P}$ such that $w(z_0)\not=0$, then $z_0$ is a zero of $e^h+\frac{\Delta^na-b}{w}$. This implies $z_0$ is a zero of $\frac{a-b}{P}-\frac{\Delta^na-b}{w}$. Thus, $N_1\le N(r,\frac{1}{\frac{a-b}{P}-\frac{\Delta^na-b}{w}})$. Since  $\frac{a-b}{P}-\frac{\Delta^na-b}{w}$ is a small function of $e^h$, we get $\frac{a-b}{P}=\frac{\Delta^na-b}{w}$. Substituting it into \eqref{2.6} yields $\frac{w}{P}=e^Q$, which  can be  be written as \eqref{2.7}. Therefore,  the remaining proof is the same as that in Subcase 2.1,  we omit it.
 
  }
\end{subcase}
}
\end{case}
\end{proof}

\subsection{Proof of Theorem \ref{th1.1}}

Since \(\Delta^n f-b\) and \( f -b\) share \(0 \) \text{CM},  we have equation \eqref{2.1}.  By Lemma \ref{lm5},  we have  \(f = a + P e^{h}\), where \(P\) is an entire function with \(\rho(P)<\rho(e^h)\),  \(h\) is a non-constant polynomial, $P,b$ and $a$ are small functions of $e^h$.

 If $\Delta^nf\equiv0$ and $\rho(a)<1$, then   Lemma \ref{lm5} yields $f = a + Aa e^{\gamma z},$ $a=b$, $e^{\gamma}=1$ and  $\Delta^n f-b=e^Q(f-b)$, where $a,b$ are polynomials, $\Delta^na\equiv 0$ and $-A=\frac{1}{e^Qe^{\gamma z}}$, $ \gamma$ are non-zero constants.  Thus, Theorem \ref{th1.1}-(ii) is proved.
 
In the following, we consider the case where  $\Delta^nf\not\equiv0$. 
By  the proof of Lemma \ref{lm5}, we get   $\rho(f)=\deg(h)\ge 1$.  
Let $w(z)=\sum_{j = 0}^{n} C_{n}^{j} (-1)^{n - j} P(z + j) e^{h(z + j)-h(z)}$,  by \eqref{2.2}, we also have \eqref{2.5}. Next, we divide  the proof into the following two cases.

\vspace{6pt}

 \setcounter{case}{0}

 \begin{case}
 \rm{
 $\rho(f)=\deg (h)>1$.

Assume that  $a\equiv b$, then using Lemma  \ref{lm2.4}  to \eqref{2.1} yields  $\deg(Q)=\deg (h)-1\ge1$. 
If $w\equiv0$, then from \eqref{2.5},  we get  $T(r,e^h)+S(r,e^h)=T(r,e^Q)$. This  is a contradiction. Therefore, $w\not\equiv0$.  Using the   identity  of Valiron-Mohon'ko \cite[Theorem 2.2.5]{Laine1}  to \eqref{2.5} implies $\Delta^na=b$. Thus \eqref{2.5} yields $\frac{w}{P}=e^Q$, which implies $\frac{\Delta^n (f-a)}{f-a}=\frac{w}{P}=e^Q$. Therefore, $\Delta^n(f-a) $ and $f-a$ share $0$ CM. Since  $\lambda(f-a)<\rho(f-a)$, by using Lemma \ref{chen}, we get $f-a=ce^{c_1z}$, where $c,c_1$ are non-zero constants. This is a contradiction to $\rho(f)>1$.
 

Now we consider the case $a\not\equiv b$. If $w\equiv0$, then from \eqref{2.5}, we get the zero of $Pe^h+a-b$ must be the zero of $\Delta^na-b$, which is a contradiction  according to $a,b,p$ are small functions of $e^h$.  Therefore $w\not\equiv0$,
by employing the same methods as in Subcase 2.2 of the proof of Lemma \ref{lm5}, we can conclude that $\frac{w}{P}=e^Q$. Therefore, $\Delta^n(f-a) $ and $f-a$ share $0$ CM.  Since  $\lambda(f-a)<\rho(f-a)$, by using Lemma \ref{chen}, we get $f-a=ce^{c_1z}$, where $c,c_1$ are non-zero constants. This is a contradiction to $\rho(f)>1$.
 
}
 \end{case}
 
\vspace{4pt}

 \begin{case}
 \rm{
 $\rho(f)=\deg (h)=1$.

  Let $h=\gamma z$, where $\gamma$ is a non-zero constant.  If $w\equiv0$, then from \eqref{2.5}, we get the zero of $Pe^h+a-b$ must be the zero of $\Delta^na-b$.
By using the second main theorem to $Pe^h$, we get $a=b$. Since $\rho(a)<\rho(f)=1$,  $\lambda(f-a)<\rho(f)$ and  $\Delta^nf $ and $f$ share $a$ CM, we use Lemma \ref{lm2.4} to  yield $a=0$. From  \eqref{2.5}, we have $f \equiv  0$, this is a contradiction. Thus $w\not\equiv0$.

If $a\equiv b$, by employing the same methods as in Subcase 2.1 of the proof of Lemma \ref{lm5}, we can conclude that $\Delta^na\equiv b\equiv a$ and \eqref{2.7} holds.
If $a$ is non-constant,by Lemma \ref{lm3} and \(\rho(a)<1\), we deduce that  there exists a finite logarithmic measure \(E\) and a small positive constant \(\epsilon\) such that \[1=\left|\frac{\Delta^{n}a}{a}\right|\leq|z|^{n(\rho(a)-1+\epsilon)}\to 0,\ as\ |z|\to\infty,  |z| = r\notin E ,\]which is impossible. Hence,   $a$ is a constant. $\Delta^na=a$ implies $a=0$, which means $b=0$. 
Therefore, $\Delta^nf $ and $f$ share $0$ CM. By   Lemma  \ref{chen}, 
we get $f=pe^{\gamma z}$, where $p$ is a non-zero constant. Then, \eqref{2.7} yields
\begin{align}
\label{2.9}
(e^{\gamma}-1)^{n}-e^Q=\sum_{j = 1}^{n}C_{n}^{j}(-1)^{n - j}e^{\gamma j}+((-1)^{n}-e^Q)=0.
\end{align}

Next, we consider the case  $a\not\equiv b$. By employing the same methods as in Subcase 2.2 of the proof of Lemma \ref{lm5}, we can conclude that $e^Q=\frac{w}{P}$, which is \eqref{2.7} . Therefore  $\Delta^n(f-a) $ and $f-a$ share $0$ CM. Since  $\lambda(f-a)<\rho(f-a)$, by using Lemma \ref{chen}, we derive the desired result
$f=a + pe^{\gamma z}, $
where $p$ is a non-zero constant. From \eqref{2.9}, we get $(e^{\gamma}-1)^n=e^Q.$

Thus, we conclude that  $f(z) = a(z) + pe^{\gamma z},$
 and $\Delta^n f-b=q(f-b)$, where \( p, \gamma, q \) are  non-zero constants and $(e^{\gamma}-1)^n=q.$ In particular, if \( a = b \), then \( a=0 \).

}    $\hfill\square$
 \end{case}

\section{{\texorpdfstring{$f'$ and $f(z+1)$ share an entire function}{f' and f(z+1) share an entire function}}}\label{sec4}
In 2014, Liu et al. \cite{Liu2} obtained the following result on the  delay-differential analogues of the Br\"uck conjecture.

\vspace{6pt}

\begin{thmx}
\label{thh}
Let $f$ be a transcendental entire function with finite order. Suppose that $f$ has a Picard exceptional value $a$ and $f'(z)$ and $f(z + 1)$ share the constant $b$ CM, then
\begin{align*}
    \frac{f'(z)-b}{f(z + 1)-b}=A,
\end{align*}
 where $A$ is a non-zero constant. Furthermore, if $b\neq0$, then $A = \frac{b}{b - a}$.
\end{thmx}

\vspace{6pt}

We find that the following example shows that Theorem \ref{thh} still holds if $a, b$ are polynomials.

\vspace{6pt}

\begin{exa}
\label{e31}
Let $f(z)=z+ze^{z}$, $a=z$, $b=z+1+\frac{ez}{1-e}$, then $\frac{f'(z)-b}{f(z+1)-b}=\frac{1}{e}$.
\end{exa}

\vspace{4pt}

Motivated by the difference analogues to Nevanlinna theory for meromorphic functions of hyper order strictly less than 1, one may ask the following question:
 
\vspace{4pt}

\begin{Que}
\label{q3}
Whether the theorem \ref{thh} still holds if the entire function $f$ is of hyper
order strictly less than 1 and the conditions on  $a,b$ are weakened?

\end{Que}

    \vspace{4pt}

If $f-a$ has finitely many zeros, then meromorphic
function $a$ is called a generalized Picard exceptional function of $f$.
Under the condition that $f$ has a generalized Picard exceptional small entire function, we give an affirmative  answer to the Question \ref{q3}  by proving the following theorem.

    \vspace{4pt}

\begin{thm}
\label{th3}
Let $f$ be a transcendental entire function with $\rho_2(f)<1$. Let $a$ and $b$ be small entire functions of $f$ such that $a$ is a generalized Picard exceptional function of $f$. Suppose  $f'(z)$ and $f(z + 1)$ share  the function $b(z)$ CM. Then,
 \[
  f(z) =a(z)+p(z)e^{\beta z} \quad \text{and} \quad  \frac{f'(z)-b(z)}{f(z+1)-b(z)}=\frac{\beta}{e^{\beta}},
 \]
  where $\beta$ is a non-zero constant and $p$ is a non-zero polynomial with $\deg (p)=k\le 1$. What's more, if $k=1$,  then \( \beta = 1 \). If $a(z+1)\not\equiv b(z)$, then $\frac{\beta}{e^{\beta}}=\frac{a'(z)-b(z)}{a(z+1)-b(z)}$. 
\end{thm}

 \vspace{4pt}

 We provide  the following  examples   to illustrate  the results of Theorem \ref{th3}.

 \vspace{4pt}

\begin{exa}

\begin{enumerate}
		\item[\rm{(1)}]Let $f(z)=ze^{ z}$, then $f'(z)=(1+z)e^z$ and $\frac{f'(z)}{f(z+1)}=\frac{1}{e}$.
		
		\item[\rm{(2)}]Let $f(z)=z+e^{\beta z}$, $b=2z+1$, where $\beta$  is a constant such that $\frac{\beta}{e^\beta}=2$. Then $f'(z)-b(z)=\beta e^{\beta z}-2z$ and $f(z+1)-b(z)=e^{\beta}e^{\beta z}-z$. Thus, $\frac{f'(z)-b(z)}{f(z+1)-b(z)}=2$
        \item[\rm{(3)}]Let $f(z)=e^{2z}$, then $\frac{f'(z)}{f(z+1)}=\frac{2}{e^2}$.
		\item[\rm{(4)}]Let $f(z)=e^{\beta z}$, where $\frac{\beta}{e^{\beta}}=1$. And  let $b(z)\not\equiv0$ be an arbitrary small entire function of $f$.  Then $f'(z)=\beta e^{\beta z}$ and $f(z+1)=e^{\beta}e^{\beta z}$. Thus, $\frac{f'(z)-b(z)}{f(z+1)-b(z)}=1$
	\end{enumerate}

\end{exa}

\begin{proof}[Proof of Theorem \ref{th3}]
Since $f-a$ has finitely many zeros,  we can assume \(f = a + p e^{h}\), where \(p\) is a is a polynomial and \(h\) is an entire function with $\rho(h)<1$.  Then $T(r,f)+S(r,f)=T(r,e^h)$ and $\rho(f)=\mu(f)$.
It is easy to see that $p,b$ and $a$ are small functions of $e^h$.  Set \(p(z)=a_kz^{k}+a_{k - 1}z^{k}+\cdots+a_{0}\), where \(k~(\ge0)\) is an integer, $a_i~ (i=0,1,\dots,k)$ are constants such that $a_k\not=0$.

Noting that \(f'(z)\) and \(f(z+1)\) share \(b\) CM, then we get
\begin{align}
\label{4.1}
    \frac{f'(z) - b}{f(z+1)- b} = e^{Q},
\end{align}
where \(Q\) is an entire function.  Subsisting \(f = a + p e^{h}\) into \eqref{4.1}, then
\begin{align}
\label{4.2}
\frac{(p'+ph')e^h+a'-b}{p(z+1)e^{h(z+1)}+a(z+1)-b}=e^Q.
\end{align}
Let $w_1=p'+ph'$ and $w_2=p(z+1)e^{h(z+1)-h(z)}$, then $w_1$ and $w_2$ are small functions of $e^h$. We rewrite \eqref{4.2} as following
\begin{align}
\label{4.3}
\frac{(e^h+\frac{a'-b}{w_1})w_1}{(e^h+\frac{a(z+1)-b}{w_2})w_2}=e^Q.
\end{align}
Next, we consider two cases.

 \setcounter{case}{0}
\begin{case}
\rm{ $a(z+1)\equiv b$.
	
	 From \eqref{4.3}, we see that the zero of $e^h+\frac{a'-b}{w_1}$ must be the zero of $w_2$. Thus, $a'=b$. Otherwise, by the Nevanlinna  second main theorem, we can get $T(r,e^h)<N(r,\frac{1}{e^h+\frac{a'-b}{w_1}})+S(r,e^h)\le  N(r,\frac{1}{w_2})+S(r,e^h)\le S(r,e^h)$, which is a contradiction. Therefore, $a'=b=a(z+1)$.

By $a(z+1)\equiv b$, $a'=b$  and  \eqref{4.3}, we have $\frac{w_1}{w_2}=e^Q$. That is $p'+ph'=e^{Q+h(z+1)-h}p(z+1)$. Thus, $Q+h(z+1)-h$  is a constant from $\rho(h)<1$ and $p$ is a polynomial. Now, we can deduce $h'=\frac{e^{Q+h(z+1)-h}p(z+1)-p'}{p}$ is a  polynomial. As $z$ tends to infinity, $h'$ becomes a constant, which implies $\deg (h)\le 1$. Without loss of generality, we can assume that $h =\beta z$, where $\beta$ is a constant. Subsisting it into
$p'+ph'=e^{Q+h(z+1)-h}p(z+1)$, then $p'+p\beta=e^Qe^{\beta}p(z+1)$. This means $e^Q=q$ is a constant  by $p$ is a polynomial. Then, comparing the coefficient of \(z^{k}\),  \(z^{k-1}\) and  \(z^{k-2}\) of both sides of  $p'+p\beta=qe^{\beta}p(z+1)$, we get
$e^{\beta }q=\beta$, $ka_k+\beta a_{k-1}=(ka_k+a_{k-1})\beta$ and $(k-1)a_{k-1}+\beta a_{k-2}=(a_k\frac{k(k-1)}{2}+a_{k-1}(k-1)+a_{k-2})\beta$. Thus, $e^{\beta}q=\beta$; if $k\ge 1$, then $k=1$, $q=\frac{1}{e}$ and $\beta=1$.

In this scenario, we conclude that
\[
f(z) =a(z)+ p(z)e^{\beta z}, \frac{f'(z)-b(z)}{f(z+1)-b(z)}=\frac{\beta}{e^{\beta}},  \quad \text{and} \quad \deg(p)=k \leq 1.
\]
Furthermore,
  if \( k = 1 \),  then \( \beta = 1 \).


\begin{case}\rm{
$a(z+1)\not\equiv b$.

 By employing the same methods as in Subcase 2.2 of the proof of Lemma \ref{lm5},  we can get $\frac{a(z+1)-b}{w_2}=\frac{a'-b}{w_1}$ and  $\frac{a'-b}{a(z+1)-b}=e^Q$. Thus, $(e^Q-1)(a(z+1)-b)=a'-a(z+1)$. Subsisting $\frac{a(z+1)-b}{w_2}=\frac{a'-b}{w_1}$ into \eqref{4.3}, then $p'+ph'=p(z+1)e^{Q+h(z+1)-h}$.  By employing the same methods as in Case 1,  we conclude that
$f(z) =p(z)e^{\beta z}+a(z),$ $ \frac{f'(z)-b}{f(z+1)-b}=\frac{\beta}{e^{\beta}}=\frac{a'-b}{a(z+1)-b},$ $ \text{and}  \deg(p) \leq 1.$
Furthermore,
if \( k = 1 \), then \( \beta = 1 \). Thus, Theorem \ref{th3}-(ii) is proved.


 The proof is completed.

}
\end{case}

}
\end{case}

\end{proof}

 \section{Application}\label{sec3}
Another special topic widely studied in the uniqueness theory is the case when two meromorphic functions $f,g$ share a set.
Given a non-constant meromorphic function $f$ in the complex plane, let $S$ be a set of meromorphic functions. We then define
$
E(S,f):=\bigcup_{a\in S}\{z:f(z)-a=0\},
$
\noindent counting multiplicities, i.e., each zero of multiplicity
$m$ will be counted $m$ times into the set $E(S,f)$. We now say
that two meromorphic functions $f,g$ share a set $S$, counting
multiplicities, provided that $E(S,f)=E(S,g)$. The first
uniqueness results for meromorphic functions making use of this
notion of sharing a set were made, to our knowledge, by
Gross \cite{{Gross}}. For some  developments in this area, see \cite[ pp. 194-199.]{Hu}.

In this section, we apply the results from Sections \ref{sec2} and \ref{sec4} to study two relationships:
(i) between an entire function \( f \) and \( \Delta^n f \), and
(ii) between \( f' \) and \( f(z + c) \),
provided that in each case the pair shares a finite set.

\subsection{\texorpdfstring{$f$ and $\Delta^n f$ share a finite set}{f and Delta^n f share a finite set}}

In this subsection, we will apply Theorem \ref{th1.1} to research the relationship between $f$ and $\Delta^n f$, under the condition that $f$ and $\Delta^n f$ ($\Delta^n f \not\equiv 0$) share a finite set.
For this purpose, we briefly review prior work on this topic.
Liu \cite{Liu}  paid attention to ${f}$ and its shifts sharing a finite set and derived the following result.

\vspace{6pt}

\begin{thmx}\cite{Liu}
\label{thd}
Suppose that \(a\) is a non-zero complex number, \(f\) is a transcendental entire function with finite order. If \(f\) and \(\Delta_{c}f\) share \(\{a, -a\}\) CM, then \(\Delta_{c}f(z)=f(z)\) for all \(z\in\mathbb{C}\).
\end{thmx}

\vspace{6pt}

In the same paper, Liu posed the following question

\vspace{6pt}
 \noindent{\bf Liu's Question} \cite[Remark 2.5]{Liu} : Let \(f\) be a transcendental entire function with finite order. And Let \(a\) and \(b\) be two small functions of \(f\) with period \(c\) such that $f$ and  \(\Delta_{c}f\) share the set \(\{a, b\}\) CM. Then, what can we say about the relationship between \(f\) and \(\Delta_{c}f\) ?

 \vspace{6pt}
  For this question, Li \cite{Li1} et al. proved the following theorem.

\vspace{6pt}

\begin{thmx}\cite{Li1}
\label{the}
Suppose that \(a\), \(b\) are two distinct entire functions, and \(f\) is a non-constant entire function with \(\rho(f)\neq1\) and \(\lambda(f)<\rho(f)<\infty\) such that \(\rho(a)<\rho(f)\) and \(\rho(b)<\rho(f)\). If \(f\) and \(\Delta_{c}f\) share \(\{a, b\}\) CM, then \(\Delta_{c}f(z)=f(z)\) for all \(z\in\mathbb{C}\).
\end{thmx}

\vspace{6pt}

 Qi et al. \cite{Qi} showed that Theorem \ref{the} still holds without the condition \(\rho(f)\neq1\).
 Guo et al. \cite{Guo}  generalized the first difference operator \( \Delta_c f \)  to the \( n \)th difference operator \( \Delta_c^n f \) in Qi's result  \cite[P.2, Main result]{Qi}.
 \vspace{6pt}

\begin{thmx}\cite{Guo}
\label{thg}
    Suppose that \(a, b\) are two distinct entire functions, and \(f\) is an entire function of hyper-order strictly less than 1 such that \(\lambda(f) < \rho(f)\), \(\rho(a) < \rho(f)\) and \(\rho(b) < \rho(f)\). If \(f\) and \(\Delta^n f(z) (\not\equiv 0)\) share the set \(\{a, b\}\) CM, then \(f(z) = A e^{\lambda z}\), where \(A, \lambda\) are two non-zero constants with \((e^{\lambda} - 1)^n = \pm 1\). Furthermore,
\begin{enumerate}
    \item[\rm{(ii)}] if \((e^{\lambda } - 1)^n = 1\), then \(\Delta^n f(z) = f(z)\);
   \item[\rm{(ii)}] if \((e^{\lambda } - 1)^n = -1\), then \(\Delta^n f(z) = -f(z)\) and \(b = -a\).
\end{enumerate}
\end{thmx}

If $\lambda(f-c)<\rho(f)$, then the meromorphic function $c$ is called a Borel exceptional function (value) of $f$. In Theorem \ref{thg}, $0$ is the Borel exceptional value of $f$; therefore, it is natural to consider the case where $f$ has a nonzero Borel exceptional function.
By using Theorem \ref{th1.1}, we   obtaine the following results.

\vspace{4pt}

\begin{thm}
\label{th2.1}
    Suppose that  \(f\) is an entire function of hyper-order strictly less than 1 and  $a, b, c$ are mall entire functions of $f$ such that $c$ is nonzero, \(\lambda(f-c) < \rho(f)\), $\rho(c)<\rho(f)$ and $a\not\equiv b$. If \(f\) and \(\Delta^n f(z) (\not\equiv 0)\) share the set \(\{a, b\}\) CM,
    then $f(z)=c(z)+pe^{\gamma z}$, where $p,\gamma$ are non-zero constants. What's  more, $(e^{\gamma}-1)^n=-1$, $\Delta^n(f-c)=-(f-c)$  and $\Delta^nc+c=a+b$.
\end{thm}

\begin{rem}
\begin{enumerate}
\item[\rm{(1)}]  Compared to Theorem~\ref{thg}, the condition on $a$ and $b$ in our theorem  differs in that we require $a$ and $b$ both to be  small functions  of $f$   without   assuming $\rho(a) < \rho(f)$ and $\rho(b) < \rho(f)$. In fact,  
under the assumptions about $a,b$ of Theorem \ref{thg}, we deduce that $a$ and $b$ are both small entire functions of $f$. In the following, we 
present the proof.

   Because \(f \) has two Borel exceptional value \(0\) and \(\infty\), we obtain that \(f \) is   regular growth with \(\rho(f)=\mu( f)\geq 1\). Therefore,  \(\rho(a) < \rho(f)\) and \(\rho(b) < \rho(f)\) yields that $a,b$ are small functions of $f$. 
		\item[\rm{(2)}] If $a=b$, then Theorem \ref{th1.1} gives the the relationship
between $f$ and $\Delta^nf$.
\item [\rm{(3)}] In comparison with Theorem \ref{thg}, we see that case $\Delta^n f\equiv f$ cannot occur. Otherwise, Theorem \ref{th2.1} would imply that $2f=\Delta^nf+f=\Delta^nc+c$, which is impossible, since $c$ is a small function of $f$.

	\end{enumerate}
\end{rem}

\vspace{6pt}

The following examples are given to  illustrate    the results of  Theorem \ref{th2.1}.
\begin{exa}
\label{e2}
(1) Let $f=z+e^{\gamma z}$, $c=b=z$, $a=0$, and $\gamma=\ln(i+1)$. Then $\Delta^2f=-e^{\gamma z}$ , $\Delta^2f-b=-e^{\gamma z}-z$, therefore \(f\) and \(\Delta^2 f(z)\) share the set \(\{a, b\}\) CM, \(\Delta^2(f-z) = -1(f-z)\).

(2)   Consider \( f(z) = e^{z\ln2} (e^{2\pi i z} + e^{6k\pi iz}) \). Obviously, \( \Delta f(z) = f(z) \). Assume that \( a, b \) are two arbitrary entire functions of order less than 1. Then \( \Delta f(z) \) and \( f(z) \) share the set \( \{a, b\} \) CM.  However, $e^{ln2z}e^{6k\pi iz}$ is a Borel exceptional function of $f$,  and the form of \( f \) does not satisfy the conclusion of Theorem \ref{th2.1}. This  example    shows that the condition $\rho(c)<\rho(f)$ is  necessary.
\end{exa}

\vspace{6pt}

When $n=1$,  then Theorem \ref{th2.1} implies $e^{\gamma}=0$, which is impossible. Thus, under the assumptions of Theorem \ref{th2.1},  the  case \(f\) and \(\Delta^n f(z) (\not\equiv 0)\) share the set \(\{a, b\}\) CM cannot occur. Therefore,  from Theorems \ref{thg} and \ref{th2.1}, we obtain the following corollary, which
         partially  solves  {\bf Liu's question }\cite[Remark 2.5]{Liu}.
         
\begin{cor}
  Suppose that  \(f\) is an entire function of hyper-order strictly less than 1 and  $a, b, c$ are mall entire functions of $f$ such that \(\lambda(f-c) < \rho(f)\), $\rho(c)<\rho(f)$ and $a\not\equiv b$. If \(f\) and \(\Delta f(z) (\not\equiv 0)\) share the set \(\{a, b\}\) CM,
    then $c\equiv0$, and $\Delta f=f$.
\end{cor}

\vspace{6pt}

\begin{proof}[Proof of Theorem \ref{th2.1}]

By the assumption \(\lambda(f-c) < \rho(f)\) and Hadamard factorization theorem, we suppose that \(f(z)=h(z)e^{\beta(z)}+c\), where \(h(z)(\not\equiv 0)\) and \(\beta\) are two entire functions satisfying
$\lambda(f-c)=\rho(h)<\rho(f)=\rho(e^{\beta}),~\rho(\beta)=\rho_2(f)<1.$ Then  $a,b,c, h$ are small functions of $e^{\beta}$ by $T(r,e^\beta)=T(r,f)+S(r,f)$. Since \(f\) and \(\Delta_c^n f\) share the set \(\{a, b\}\) CM,  we get
\begin{align}
\label{3.1}
\frac{(\Delta_c^n f - a)(\Delta_c^n f - b)}{(f - a)(f - b)} = e^{\alpha},
\end{align}
where \(\alpha\) is an entire function.

Substituting  the forms of \(f\) and \(\Delta_c^n f\) into \eqref{3.1} yields that
\begin{align}
\label{3.2}
&\left(\left[\sum_{i = 0}^{n}(-1)^{n - i}\binom{n}{i}h(z + ic)e^{\beta(z + ic)-\beta(z)}\right]e^{\beta(z)}+\Delta^nc-a\right)\\ \nonumber
&\left(\left[\sum_{i = 0}^{n}(-1)^{n - i}\binom{n}{i}h(z + ic)e^{\beta(z + ic)-\beta(z)}\right]e^{\beta(z)}+\Delta^nc-b\right)\\ \nonumber
&= e^{\alpha}(h(z)e^{\beta(z)}+c-a)(h(z)e^{\beta(z)}+c-b).
\end{align}
Since $a\not\equiv b$, we note that \(c - a\) and \(c - b\) are not both zero. Without loss of generality,  we   suppose that $c-a\not\equiv0$. Set \(\omega=\frac{\Delta^nf-\Delta^nc}{e^\beta}=\sum_{i = 0}^{n}(-1)^{n - i}\binom{n}{i}h(z + ic)e^{\beta(z + ic)-\beta(z)}\). 
 
 We claim $\omega\not\equiv0$. Otherwise,
if $\omega\equiv0$, then \eqref{3.2} reduces to
\begin{align*}
\frac{(\Delta^nc-a)(\Delta^nc-b)}{(h(z)e^{\beta(z)}+c-a)(h(z)e^{\beta(z)}+c-b)}=e^\alpha.
\end{align*}
Then the zero of $he^\beta+c-a$ must be the zero  of $(\Delta^nc-a)(\Delta^nc-b)$. However, $(\Delta^nc-a)(\Delta^nc-b)\not\equiv0$. Otherwise, $f\equiv a$ or $f\equiv b$, which is impossible.  By  the  Nevanlinna second main theorem, we get $T(r,he^{\beta})\le N(r,\frac{1}{(\Delta^nc-a)(\Delta^nc-b)})+S(r,e^\beta)\le S(r,e^\beta)$, this is a contradiction.   Hence, $\omega\not\equiv0$. By  difference logarithmic derivative lemma,   we deduce $\omega$ is a small function of $e^\beta$.

We rewrite \eqref{3.2} as following:
\begin{align}
\label{3.3}
e^{\alpha}=\frac{\omega^{2}\left[e^{\beta}+\frac{\Delta^nc-a}{\omega}\right]\left[e^{\beta}+\frac{\Delta^nc-b}{\omega}\right]}{h^{2}\left[e^{\beta}+\frac{c-a}{h}\right]\left[e^{\beta}+\frac{c-b}{h}\right]}.
\end{align}
 Then the zeros of $e^{\beta}+\frac{c-a}{h}$ must be the zeros  of  $\left[e^{\beta}+\frac{\Delta^nc-a}{\omega}\right]\left[e^{\beta}+\frac{\Delta^nc-b}{\omega}\right]$ and $\omega^2$.  Below, we denote by \(N_1\) the  counting function of those common zeros of $e^{\beta}+\frac{c-a}{h}$  and $e^{\beta}+\frac{\Delta^nc-a}{\omega}$. Similarly, denote by \(N_2\) the counting function of those common zeros of $e^{\beta}+\frac{c-a}{h}$  and $e^{\beta}+\frac{\Delta^nc-b}{\omega}$. Note that \(h\) is a small function with respect to \(e^{\beta}\); applying the second fundamental theorem to \(e^{\beta}\) yields that
\begin{align}
\label{3.4}
T(r,e^{\beta})\le N\left(r,\frac{1}{e^{\beta}+\frac{c-a}{h}}\right)+S(r,e^{\beta}) = N_1+N_2+S(r,e^{\beta}),
\end{align}
which implies that either \(N_1 \neq S(r,e^{\beta})\) or \(N_2\neq S(r,e^{\beta})\). Next, we consider these two cases.

 \setcounter{case}{0}
\begin{case}\rm{

\(N_1 \neq S(r,e^{\beta})\).
  
  Let $z_0$ be the
common zero of $e^{\beta}+\frac{c-a}{h}$  and $e^{\beta}+\frac{\Delta^nc-a}{\omega}$, then $z_0$ is the zero of $\frac{c-a}{h}-\frac{\Delta^nc-a}{\omega}$.
Thus, if $\frac{c-a}{h}-\frac{\Delta^nc-a}{\omega}\not\equiv0$, then $N_1\le N(r,\frac{1}{\frac{c-a}{h}-\frac{\Delta^nc-a}{\omega}})\le S(r,e^\beta)$, which is a contradiction to \(N_1 \neq S(r,e^{\beta})\). Therefore, $\frac{c-a}{h}=\frac{\Delta^nc-a}{\omega}$. 
 Substituting it into the equation \eqref{3.3}, then
\begin{align}
\label{4.41}
e^{\alpha}=\frac{\omega^{2}\left[e^{\beta}+\frac{\Delta^nc-b}{\omega}\right]}{h^{2}\left[e^{\beta}+\frac{c-b}{h}\right]}.
\end{align}

If $c\equiv b$, then   the zero of $e^{\beta}+\frac{\Delta^nc-b}{\omega}$
 must be the zero of $h$. 
 By using the  second main theorem to $e^{\beta}$, we get $\Delta^nc=b=c$.
From \eqref{4.41},  we have $\frac{w^2}{h^2}=e^{\alpha}$.  Therefore, $\frac{\Delta^n(f-c)}{f-c}=\frac{\omega}{h}=e^{\frac{\alpha}{2}+k\pi i}$, where $k$ is an integer.
This implies $\Delta^a(f-c)$  share $0$ CM with $f-c$. Since $f-c=he^\beta$ and $\lambda(f-c)<\rho(f-c)$,
by Theorem \ref{th1.1}, we get $f-c =pe^{\gamma z},$   where \( p, \gamma\) are  non-zero constants and $(e^{\gamma}-1)^n=1.$ Thus, $\rho(c)<\rho(f)=1$.  

 Since $\Delta^nc=c$, by Lemma \ref{lm3} and \(\rho(c)<1\), there exists a finite logarithmic measure \(E\) and a small positive constant \(\epsilon\) such that for \(|z| = r\notin E\),
\[
1=\left|\frac{\Delta^{n}c}{c}\right|\leq|z|^{n(\rho(c)-1+\epsilon)}\to 0,\ as\ |z|\to\infty,
\]
which is impossible unless $c$ is a constant. $\Delta^n=c$ implies $c=0$.
 This contradicts our assumption that 
$c\not\equiv0$. 
 

If $c\not\equiv b$,  using the Nevanlinna second main theorem to $e^{\beta}$, we obtain $N(r,\frac{1}{e^{\beta}+\frac{c-b}{h}})=T(r,e^{\beta})+S(r,e^{\beta})$. 
Let $z_0$ be  the
common zero of $e^{\beta}+\frac{c-b}{h}$  and $e^{\beta}+\frac{\Delta^nc-b}{\omega}$, then $z_0$ is the zero of $\frac{c-b}{h}-\frac{\Delta^nc-b}{\omega}$.
Thus, $\frac{c-b}{h}=\frac{\Delta^nc-b}{\omega}$. 
From \eqref{4.41},  we have $\frac{w^2}{h^2}=e^{\alpha}$.  Therefore, $\frac{\Delta^a(f-c)}{f-c}=\frac{\omega}{h}=e^{\frac{\alpha}{2}+k\pi i}$, where $k$ is an integer.
 Using a similar method as in the case $c\equiv b$, we can obtain $\rho(c)<\rho(f)=1$. 

  We have already assumed that $a \not\equiv c$, from $\frac{c-b}{h}=\frac{\Delta^nc-b}{\omega}$ and $\frac{c-a}{h}=\frac{\Delta^nc-a}{\omega}$, we get  
  $\frac{\Delta^nc-a}{c-a}=\frac{\Delta^nc-b}{c-b}$. This means $a(\Delta^nc-c)=b(\Delta^nc-c)$. Since $a\not\equiv b$, we get $\Delta^nc=c$. Using a similar method as in the case $c\equiv b$  impiles $c=0$.
This contradicts our assumption that 
$c\not\equiv0$.


}
\end{case}

\begin{case}
\rm{
\(N_2\neq S(r,e^{\beta})\).
 
 Using the same approach as in Case 1, we obtain $\frac{c-a}{h}=\frac{\Delta^nc-b}{\omega}$.
Substituting it into the equation \eqref{3.3}, then
\begin{align}
\label{4.42}
e^{\alpha}=\frac{\omega^{2}\left[e^{\beta}+\frac{\Delta^nc-a}{\omega}\right]}{h^{2}\left[e^{\beta}+\frac{c-b}{h}\right]}.
\end{align}

If $c\equiv b$,  we get $\Delta^nc=a$ from \eqref{4.42}. $\frac{c-a}{h}=\frac{\Delta^nc-b}{\omega}$ impels $\frac{\omega}{h}=-1$.
 Thus, 
 $\frac{\Delta^nf-\Delta^nc}{f-c}=-1
$.

When $c\not \equiv b$, using the same approach as in Case 1, we obtain then 
 $ \frac{\Delta^nc-a}{\omega}=\frac{c-b}{h}$ from \eqref{4.42}.
  $\frac{\Delta^nc-b}{c-a}=\frac{\Delta^nc-a}{c-b}$ implies $a+b=\Delta^nc+c$. Thus, $\frac{\Delta^nc-b}{c-a}=\frac{w}{h}=-1$, we obtain  
 $\frac{\Delta^nf-\Delta^nc}{f-c}=-1
$.

In summary, we  get  $\Delta^n(f-c)$ share $0$ CM with $f-c$. Since $f-c=he^\beta$,  we have $\lambda(f-c)=\rho(h)<\rho(e^\beta)=\rho(f-c)$. By Theorem \ref{th1.1}, we get $f-c =pe^{\gamma z},$  where \( p, \gamma,  \) are  non-zero constants and $(e^{\gamma}-1)^n=-1.$

 Thus, Theorem \ref{th2.1}  is proved.
}
\end{case}
\end{proof}

\subsection{\texorpdfstring{$f'$ and $f(z+1)$ share a finite set}{f' and f(z+1) share a finite set}}

In this subsection, we  employ Theorem \ref{th3} to investigate the relationship between $f'(z)$ and $f(z+1)$, under the condition that $f'$ and $f(z+1)$ share a finite set, and obtain the following result.

\vspace{4pt}

\begin{thm}
\label{th4.2}
    Suppose that  \(f \) is a transcendental entire function of hyper-order strictly less than 1 and  $a, b, c$ are mall entire functions of $f$ such that $f-c $ has finite many zeros and $a\not\equiv b$. If \(f'\) and \(f(z+c)\) share the set \(\{a, b\}\) CM,
    then  $f(z)=c(z)+p(z)e^{\gamma z}$, where $\gamma$, $p$ are   nonzero constant  and
one of the following cases holds.
	\begin{enumerate}
		\item[\rm{(i)}] $\frac{\gamma}{e^\gamma}=1$, $c=0$ and $f'(z)=f(z+1)$.
		\item[\rm{(ii)}] $\frac{\gamma}{e^\gamma}=-1$, $(f-c)'=-[f(z+1)-c(z+1)]$, $c'(z)+c(z+1)=a+b$.
	\end{enumerate}
\end{thm}

\vspace{4pt}

\begin{rem}
    If $a=b$, then Theorem \ref{th3} gives the relationship between $f'(z)$ and $f(z+1)$.
\end{rem}

\vspace{4pt}

If $\frac{\gamma}{e^\gamma}=1$ and $c=0$, then  $f'(z)=f(z+1)$ from  $f(z)=pe^{\gamma z}$. Therefore,  \(f'\) and \(f(z+c)\) share the set \(\{a, b\}\) CM. That is  the example of Theorem \ref{th4.2} (i).  The following example is given to show that case (ii) of Theorem \ref{th4.2} may occur.
\vspace{4pt}

\begin{exa}
Let $f(z)=z+e^{\gamma z}$, $a=z$ and $b=2$ such that $e^{\gamma}=-\gamma$. Then $f'(z)=1+e^{\gamma z}\gamma$, $f(z+1)=e^{\gamma} e^{\gamma z}+z+1$, $\frac{(f'-a)(f'-b)}{(f(z+1)-a)(f(z+1)-b)}=\frac{(\gamma e^{\gamma z}+1-z)(\gamma e^{\gamma z}-1)}{(e^{\gamma}e^{\gamma z}+1)(e^{\gamma}e^{\gamma z}+z-1)}=-1$.
\end{exa}

\begin{proof}[Proof of Theorem \ref{th4.2}]
    \rm{
Since \(f'\) and \(f(z+c)\) share the set \(\{a, b\}\) CM,  then
\begin{align}
\label{4.11}
\frac{(f' - a)(f' - b)}{(f(z+1) - a)(f(z+1) - b)} = e^{\alpha},
\end{align}
where \(\alpha\) is an entire function.
By the assumption $f-c$ has finitely many zeros and Hadamard factorization theorem,  suppose that \(f(z)=h(z)e^{\beta(z)}+c(z)\), where \(h(\not\equiv 0)\)  is a polynomial and \(\beta\) is an entire functions satisfying
$~\rho(\beta)=\rho_2(f)<1.$  Then $a,b,c, h$ are small functions of $e^{\beta}$ by $T(r,e^\beta)=T(r,f)+S(r,f)$.

Substituting the forms of \(f'\) and \(f(z+1)\) into \eqref{4.11} yields that
\begin{align}
\label{4.12}
\frac{[(h'+h\beta')e^\beta+c'-a][(h'+h\beta')e^\beta+c'-b]}{[h(z+1)e^{\beta(z+1)}+c(z+1)-a][h(z+1)e^{\beta(z+1)}+c(z+1)-b]}=e^\alpha.
\end{align}

Let $w_1=h'+h\beta'$ and $w_2=h(z+1)e^{\beta(z+1)-h(z)}$, then we rewrite \eqref{4.12} as following
\begin{align}
\label{4.13}
\frac{(e^{\beta}+\frac{c'-a}{w_1})(e^{\beta}+\frac{c'-b}{w_1})w_1^2}{(e^{\beta}+\frac{c(z+1)-a}{w_2})(e^{\beta}+\frac{c(z+1)-b}{w_2})w_2^2}=e^{\alpha}.
\end{align}

The following proof is similar to that of Theorem \ref{th2.1},  we  can  employ Theorem \ref{th3} to obtain  that $f(z)=c(z)+p(z)e^{\gamma z}$, where $\gamma$ is a constant, $p$  is a polynomial of degree $k\le 1$,  and
one of the following cases holds.
	\begin{enumerate}
		\item[\rm{(i)}] $\frac{\gamma}{e^\gamma}=1$, $c=0$ and $f'(z)=f(z+1)$.
		\item[\rm{(ii)}] $\frac{\gamma}{e^\gamma}=-1$, $(f-c)'=-[f(z+1)-c(z+1)]$, $c'(z)+c(z+1)=a+b$.
	\end{enumerate}
What's more, if $k=1$,  then \( \gamma = 1 \).  However,  since $\frac{\gamma}{e^\gamma}\not=1$ or $-1$  when $\gamma=1$,  we  get $k=0$, which means $p$ is a constant.

}
\end{proof}

\section*{Declarations}
\begin{itemize}
\item \noindent{\bf Funding}
Jianren Long was supported by the National Natural Science Foundation of China (Grant No. 12261023, 11861023) and Xuxu Xiang was supported by Graduate Research Fund Project of Guizhou Province (2025YJSKYJJ107).

\item \noindent{\bf Conflicts of Interest}
The authors declare that there are no conflicts of interest regarding the publication of this paper.

\item\noindent{\bf Author Contributions}
All authors contributed to the study conception and design. All authors read and approved the final manuscript.

\end{itemize}

\end{document}